\pdfoutput=1

\documentclass[12pt]{amsart}
\usepackage{amsmath,amscd,amsfonts,amssymb,latexsym}

\usepackage[all]{xy}
\usepackage{tikz}
\usepackage{dirtytalk}
\usepackage{upgreek}
\usepackage{dsfont}
\def\h{\mathds{h}}

\newtheorem{theorem}{Theorem}[section]
\newtheorem{proposition}[theorem]{Proposition}

\newtheorem{corollary}[theorem]{Corollary}
\newtheorem{remark}[theorem]{Remark}

\newtheorem{example}[theorem]{Example}

\def\C{\mathbb{C}}
\def\Q{\mathbb{Q}}
\def\T{{\mathbb{T}}}

\def\wTT{{\widetilde{\mathbb{T}}}}
\def\tt{{\bf t}}
\def\hh{{\bf h}}

\def\Z{\mathbb{Z}}
\def\P{\mathbb{P}}

\def\O{{\mathcal  O}}

\def\tx{{\tilde x}}
\def\vt{\vartheta}

\def\tt{{\mathbf t}}
\def\ii{{\rm i}}
\def\Ell{{\mathcal E}\ell\ell}
\DeclareMathOperator\e{e}
\DeclareMathOperator\del{\Delta}
\DeclareMathOperator\Del{\delta\!}
\def\nuu{{\nu}}
\def\muu{{\mu}}

\title[Elliptic classes, McKay correspondence and theta identities]
{\Large Elliptic classes, McKay correspondence  \\\vskip 15pt and theta identities}
\author{Ma{\l}gorzata Mikosz\and Andrzej Weber}
\address{ Warsaw University of Technology\\
 ul.~Koszykowa 75, 00-662, Warszawa, Poland}
\email{emmikosz@mini.pw.edu.pl}
\address{
Institute of Mathematics, University of Warsaw\\
 Banacha 2, 02-097 Warszawa, Poland}
\email{aweber@mimuw.edu.pl}
\thanks{The second author is supported by the Polish National Science
  Center project Algebraic Geometry: Varieties and Structures,
  2013/08/A/ST1/00804 and  2016/23/G/ST1/04282 (Beethoven~2)}

\begin{document}
\baselineskip 15pt

\begin{abstract}We revisit the construction of elliptic class given by Borisov and
Libgober for singular algebraic varieties. Assuming torus action we adjust
the theory to the equivariant local situation.  We study theta function
identities having a geometric origin. In the case of quotient singularities
$\C^n/G$, where $G$ is a finite group the theta identities arise from
McKay correspondence. The symplectic singularities are of special
interest. The Du Val surface singularity $A_n$ leads to a remarkable
formula.
\end{abstract}
\maketitle
 The theory of theta functions is a classical subject of analysis and algebra. It had a prominent role in 19$^{\rm th}$ century mathematics, as one can see reading the monograph about Algebra \cite{We}. Nowadays it seems that the intriguing combinatorics related to theta functions has been put aside. Nevertheless there are modern sources treating the subject of theta functions in a wider context. For example the Mumford three volume book \cite{Mu} is devoted to the theta function. The mysterious pattern of Riemann relation is described in the initial chapters. This pattern repeats on and it is present in  our formulas described  in Theorem \ref{A1main}.
\medskip

At the same time from mid-seventies the theta function found an application in algebraic topo\-logy, in particular in the cobordism theory. The formal group laws generated by elliptic curves and the associated elliptic genera,  allowed to define elliptic cohomology  by a Conner-Floyd type theorem. The  collection of articles in LNM 1326 and in particular \cite{Lan} is a good reference for that approach. The article of Segal written for the Bourbaki seminar \cite{SegalEll} is an accessible survey of the beginnings of this theory.
\medskip

In the beginning of 2000' theta function was applied by Borisov and Libgober \cite{BoLi0} to construct an elliptic genus of singular complex algebraic varieties.  Totaro  \cite{Totaro} links this construction  with previous work on cobordisms.
The elliptic genus is defined in terms of resolution of singularities, but it does not depend on the particular resolution. The elliptic genus is the degree (i.e.~the integral) of some cohomology characteristic class called the {\it elliptic class}, denoted by $\Ell(-)$. That class is the main protagonist of our paper. The elliptic class is defined in terms of the theta function applied to the Chern roots of the tangent bundle of the resolution. Independence on the resolution translates to some relations involving the theta function. For example, the Fay trisecant relation \cite{Fay}, \cite{GTL} corresponds to a blowup of a surface at one point.
\medskip

The idea to study global invariants via contributions of singularities is common in mathematics and originates from Poincar\'e-Hopf theorem. It re-appeared in the work of Atiyah and Singer on the equivariant index theorem. In the presence of a torus action   the Atiyah-Bott-Berline-Vergne localization techniques apply. Each fixed point component gives a local summand of the global invariant.
The local equivariant Chern-Schwartz-MacPherson classes were studied in \cite{WeChern} and the local contributions to the Hirzebruch class were described in \cite{WeSel}. The role of the local contributions to the elliptic class in the Landau-Ginzburg model is demonstrated in \cite{LibLG}. Also, local computation is the key ingredient of the work on  elliptic classes of Schubert varieties in the generalized flag variety, \cite{RW}, \cite{KRW}.
\medskip

In the present paper we adjust the theory of Borisov-Libgober to the local equivariant situation. The initial step was done by Waelder \cite{Wae}, but we believe that our approach and formalism makes the theory accessible and convenient for further application.
Recently the theta function has served as a basic brick in construction of diverse objects, such as representations of quantum groups \cite{GTL}, weight function (in integrable systems) \cite{RTV} and stable envelopes \cite{AO}. The works \cite{RW}, \cite{KRW} form a  bridge between these theories and strictly geometric theory of characteristic classes.
Here we  present the Borisov-Libgober elliptic class in a way which fits to the context of the mentioned papers.
In \S\ref{explanation} we also present a conceptual approach to the elliptic class as the equivariant Euler class in a version of elliptic cohomology. This point of view   appeared in \cite{RW}. We take an opportunity to extend and clarify this approach. We explain the normalization constants, which allow to place the elliptic class in the common formalism used in algebraic topology.
\medskip

Further we study the elliptic classes of quotient singularities. In the local context these are the quotients $V/G$, where $V$ is a vector space and $G\subset GL(V)$ is a finite subgroup. The quotient $V/G$ is considered as a variety with the $\C^*$--action, coming from the scalar multiplication on $V$. Occasionally the quotient admits an action of a larger torus. If $G$ acts diagonally on $V=\C^n$, then whole torus $\T=(\C^*)^{n}$ acts on the quotient.
The elliptic classes of global quotients were studied by Borisov and Libgober in the paper on McKay correspondence \cite{BoLi}. Their main result is as follows. Suppose that $V$ is a complete variety, then the elliptic genus of $X=V/G$ is equal to the so-called
{\it orbifold elliptic genus}. The later is defined as the degree of a certain class, called the {\it orbifold class}, living on an equivariant resolution of $V$ and defined in terms of the action of the group $G$. This is the main result of \cite{BoLi}, but in the course of the proof it is checked that the orbifold class well behaves with respect to modifications of the resolution of the pair $(V,G)$.
\medskip

We wish to discuss the McKay correspondence showing examples. In its full generality the definitions are quite involving. We assume that $V$ is a vector space with $G$ acting linearly, $X=V/G$.
If $G\subset SL(V)$, then  the canonical divisor of $X$ is trivial. Sometimes $X$ admits a crepant resolution, i.e. a map $Y\to X$ from a smooth variety with (in this case) trivial canonical bundle.
McKay correspondence is a general rule which says that geometric invariants of $Y$ can be expressed in terms of group properties of $G$ and its representation $V$. The equivariant version of McKay correspondence for elliptic classes was established in \cite[Theorem 7]{Wae}).
It becomes particularly apparent when applied to the quotients of representations $V/G$, see \cite[\S6]{DBW}.
The original proof given in \cite{BoLi} is based on the analysis of the toric singularities. The Lemma 8.1 of  \cite{BoLi} states the strongest form of an identity involving the theta function, but its formulation is considerably complicated. The formula becomes much simpler for the symplectic quotients. For du Val singularities it can be related to combinatorics of Dynkin diagram.
We would like to state explicitly what McKay correspondence says for symplectic quotient singularities.
Among other examples we study symplectic quotients of $\C^2$, that is du Val singularities. Suppose
$\Z_n\subset SL_2(\C)$ consists of diagonal matrices ${\rm diag}(\xi^k,\xi^{-k})$, $\xi= e^{2\pi \ii/n}$.
The equality of elliptic classes (computed via resolution and the orbifold class) of $\C^2/\Z_n$ implies Theorem \ref{A1main}, which is an identity involving Jacobi theta function. The formula  specializes to the following simplified form
\begin{equation}\label{wstepid}n\, \Delta(n x,z)\Delta(-nx,z)=\tfrac1n\sum_{k,\ell=0}^{n-1}\Delta\left(\tfrac{k-\ell\tau}n+x,z\right)
\Delta\left(\tfrac{\tau\ell-k}n-x,z\right)\,,\end{equation}
where
$$\Delta(a,b)=\frac{\theta_\tau(a+b)\theta'_\tau(0)}{\theta_\tau(a)\theta_\tau(b)}\,.$$
Similar identities can be obtained for various quotient singularities. Any quotient singularity resolution gives rise to a theta identity. But except few cases the resulting formulas have complexities not allowing to present them in a compact form.
\medskip

The example of $A_{n-1}$ singularity given above is  particularly appealing. The related theta function identity is interesting on its own. In addition the elliptic class is self-dual, in the sense that if we exchange the \say{equivariant parameter} $x$ with the  \say{dynamical parameter} $z$ the class only changes its sign. A similar effect of duality appears in the work on stable envelopes, \cite{RSVZ}.
\medskip

The identities we consider here involve the Jacobi theta function $\theta_\tau$, but can be specialized (taking $\tau\to \ii\infty$) to some trigonometric identities. One of these is the following one:
$$\frac 1n\sum_{k=0}^{n-1}\frac{1}{1-2\cos(\frac{2k\pi }n) T+T^2}=\frac{1-T^{2 n}}{\left(1-T^2\right)
   \left(1-T^{n}\right)^2}\,.$$
We encourage the reader to give an elementary proof.
It seems that the above formula cannot be reduced to standard trigonometric identities.
\medskip

The second author would like to thank Jaros{\l}aw Wi{\'s}niewki for multiple conversations on the quotient singularities and for his good spirit in general.

\tableofcontents

\section{Basic functions}
\subsection{Theta function}
Let us introduce the notation
$$\e(x)=e^{2\pi \ii x}\,.$$
For $\upsilon,\, \tau\in \C$,  ${\rm im}(\tau)>0$
let $q_0=\e(\tau/2)=e^{\pi \ii \tau}$. We define the theta function $\theta_\tau(\upsilon)$ as in \cite{Cha}:
\begin{align*}\theta_\tau(\upsilon)&=\frac 1\ii\sum_{n=-\infty}^\infty(-1)^n q_0^{ (n+\frac12)^2}\e((2n+1)\upsilon/2)=\\
&=2\sum_{n=0}^\infty(-1)^n q_0^{ (n+\frac12)^2}\sin((2n+1)\pi \upsilon)
\,,\end{align*}
see  \cite[\S24 (4)]{We} or \cite[Ch.~V.1]{Cha}.
According to the Jacobi product formula \cite[Ch~V.6]{Cha}
\begin{equation}
\theta_\tau(\upsilon)=2q_0^{\frac14} \sin (\pi \upsilon)
\prod_{l=1}^{\infty}(1-q_0^{2l})
 \big(1-q_0^{2l} \e(\upsilon )\big)\big(1-q_0^{2l} /\e(\upsilon)\big)\,.
\end{equation}
In \cite{BoLi} the variable $q=q_0^2=e^{2\pi\ii \tau}$ is used, it also fits to the convention of \cite{RTV}, \cite{RW}. Therefore further we will express our formulas in that variable.
\begin{equation}\label{Jacobi-prod-for}
\theta_\tau(\upsilon)=2q^{\frac18} \sin (\pi \upsilon)
\prod_{l=1}^{\infty}(1-q^{l})
 \big(1-q^{l} \e(\upsilon)\big)\big(1-q^{l}/ \e(\upsilon)\big)\,.
\end{equation}
Here the symbol $q^{\frac18}$ simply means $\e(\tau/8)=e^{\pi \ii \tau/4}$.
The function theta $\theta_\tau$ is odd
$$\theta_\tau(-\upsilon)=-\theta_\tau(\upsilon)\,,$$
and it satisfies the quasi-periodic identities
$$\theta_\tau(\upsilon+1)=-\theta_\tau(\upsilon)\,,$$
\begin{equation}\label{deltatau0}\theta_\tau(\upsilon+\tau)=-q^{-1/2}\, \e(- \upsilon)\, \theta_\tau(\upsilon)\,.\end{equation}
Moreover
$$\theta_{\tau+1}(\upsilon)=\sqrt{\ii}\,\theta_\tau(\upsilon)\,,$$
where $\sqrt{\ii}=e^{\pi \ii/4}$,
$$\theta_{-1/\tau}\left(\tfrac\upsilon\tau\right)=\sqrt{\tfrac\tau \ii} \,\e\!\left(\tfrac{\upsilon^2}{ 2 \tau}\right) \theta_\tau(\upsilon)\,.$$
\medskip

The variable $\tau$ is treated as parameters and we will omit it  later.
It is convenient to use the multiplicative notation for variables.
Let
\begin{align*}\vt(x)&=x^{1/2}(1-x^{-1})\prod_{\ell=1}^\infty(1-q^\ell x)(1-q^\ell/x)\\
&=x^{1/2}(1-x)^{-1}\, (x;q)_\infty\, (x^{-1},q)_\infty\,.
\end{align*}
Here
$$(x;q)_\infty=\prod_{\ell=0}^\infty(1-q^\ell x)=(1-x)(1-qx)(1-q^2x)\dots$$
is the $q$-Pochhammer symbol. The function $\vt$ is defined on the double cover of $\C^*$, since we have $x^{1/2}$ in the formula.
We have
\begin{equation}\label{constant}\theta(\upsilon)={\rm const}\,  \vt(\e(\upsilon))\,,\qquad {\rm const}=\tfrac 1 \ii\,q^{\tfrac18}(q;q)_\infty\end{equation}
The constant term of $\theta$ is irrelevant for us. In the formula for the elliptic class we prefer to use $\vt$ rather than $\theta$ notation.
The function $\theta$ can be treated as a section of the line bundle ${\mathcal O_E}([0])$ over the elliptic curve $E=\C/\langle 1,\tau\rangle$ defined by $\tau$.
\bigskip

\begin{remark}\rm Wolfram--Mathematica package has implemented the Jacobi theta function with the following convention:
$$\theta_\tau(\upsilon)=\ii\; \rm{EllipticTheta}[1,\pi\upsilon,q_0]\,,$$
where $q_0=e^{\pi \ii\tau}$.
In MAGMA package
$$\theta_\tau(\upsilon)=\rm{JacobiTheta}(q_0,\pi\upsilon)\,.$$ The identities discussed in this paper were checked numerically to make sure that the exponents and the conventions agree.
\end{remark}

\subsection{The function $\del_\tau(a,b)$ and $\delta_\tau(a,b)$}
In the definition of elliptic characteristic classes the  quotient of the form
$$\del_\tau(a,b)=\frac{\theta'_\tau(0)\,\theta_\tau(a+b)}{\theta_\tau(a)\,\theta_\tau(b)}$$
appears.  The normalizing factor
  \eqref{constant}, which is independent of $\upsilon$, cancels out. This factor is omitted  e.g.~in \cite{AO}, \cite{RTV}, \cite{RW}, \cite{KRW}.
The meromorphic function $\del_\tau$ is odd
\begin{equation}\del_\tau(-a,-b)=-\del_\tau(a,b)\,,\end{equation}
it satisfies the following quasi-periodic identities:
$$\del_\tau(a,b)=\del_\tau(b,a)\,,$$
\begin{equation}\label{per1}\del_\tau(a+1,b)=\del_\tau(a,b)\,,\end{equation}
\begin{equation}\label{per2}\del_\tau(a+\tau,b)=\e(-b)\, \del_\tau(a,b)\,.\end{equation}
Moreover
$$\del_{\tau+1}(a,b)=\del_\tau(a,b)\,,$$
$$\del_{-1/\tau}(a,b)=\tau \e\left(\tfrac{ab}\tau\right)\,\del_\tau(a,b)\,.$$
The function $\del_\tau$ defines a section of a bundle over $E^2$.
\bigskip

We prefer the multiplicative notation
We will write
 $\delta(\tfrac AB,\tfrac CD)$ instead of $\del(a-b,c-d)$:
$$\theta(a-b)={\rm const}\,\vt(\tfrac AB)\,,\qquad\tfrac1{2\pi \ii}\Delta(a-b,c-d)=\delta(\tfrac AB,\tfrac CD)$$
when $$A=\e(a)=e^{2\pi \ii a},\quad B=\e(b)=e^{2\pi \ii b},\quad C=\e(c)=e^{2\pi \ii c},\quad D=\e(d)=e^{2\pi \ii d}.$$
The function $\delta$ is well defined on $\C^*\times \C^*$ since the fractional powers cancel out. It has the following quasi-periodic properties
\begin{equation}\label{deltatau}\delta(qA,B)=B^{-1}\delta(A,B)\,,\qquad \delta(A,qB)=A^{-1}\delta(A,B)\,.\end{equation}

\section{The $A_{n-1}$ identity}
Resolutions of quotient singularities give rise to certain identities for theta function. We will explain this mechanism in the subsequent sections, but first we would like to give an example of the cyclic symplectic quotient $\C^2/\Z_n$, i.e.~the singularity $A_{n-1}$. It leads to the following interesting identity:
\begin{theorem} \label{A1main}  Let $n>0$ be a natural number. Let $t_1$, $t_2$, $\muu_1$ and $\muu_2$ be indeterminate. The following two meromorphic functions in four variables (and $\tau$) are equal:
\begin{align*}(*)&\quad A_n(t_1,t_2,\muu_1,\muu_2)=
\sum_{k=1}^{n}\Del\left( \tfrac{t_1^{n-k+1}}{t_2^{k-1}},\muu_1^{k}\muu_2^{n-k}\right)
\Del\left(\tfrac{t_2^{k}}{t_1^{n-k}},{\muu_1^{k-1}}{\muu_2^{n-k+1}}\right),\\
(**)&\quad B_n(t_1,t_2,\muu_1,\muu_2)=\frac1n \sum_{k=0}^{n-1}
\sum_{\ell=0}^{n-1}\big(\tfrac{\muu_2}{\muu_1}\big)^\ell\Del\left(\e\left(\tfrac{k-\ell\tau}n\right)t_1,\muu_1^n\right)
\Del\left(\e\left(\tfrac{\tau\ell-k}n\right)t_2,\muu_2^n\right).\end{align*}
\end{theorem}
In particular when we specialize the identity to the case $t_1=t_2^{-1}=t$, $\muu_1=\muu_2$, and setting $\h=\muu_1^n$, we obtain
$$n\, \delta(t^n,\h)\delta(t^{-n},\h)=\tfrac1n\sum_{k,\ell=0}^{n-1}\Del\left(\e\left(\tfrac{k-\ell\tau}n\right)t,\h\right)
\Del\left(\e\left(\tfrac{\tau\ell-k}n\right)t^{-1},\h\right)\,,$$
which is the identity \eqref{wstepid} written in the multiplicative notation.
To visualize the sum (*) we can apply the following quiver, which is the Goreski-Kottwitz-MacPherson (GKM) graph of the minimal resolution of the quotient singularity $A_{n-1}$. The quiver (oriented graph) is the following: \begin{itemize}
\item $n$  vertices indexed by $V=\{1,2,\dots,n\}$,
\item $n+1$ arrows indexed by \say{weights}.
For a vertex $k\in V$ there is one arrow outgoing with the weight ${\rm out}(k)=(n-k,k)$ and one incoming with the weight ${\rm in}(k)=(n-k+1,k-1)$. The external arrows have loose ends.\end{itemize}
\def\strzd#1{\hskip-1pt\mathop{\xrightarrow{\hspace*{1.2cm}}}\limits^{#1}\hskip-1pt}
\def\strzk#1{\hskip-1pt\mathop{\xrightarrow{\hspace*{0.7cm}}}\limits^{#1}\hskip-1pt}

\begin{equation}\label{GKM}\strzk{(n,0)}\boxed{1}
\strzd{(n-1,1)}\boxed{2}\strzd{(n-2,2)}
~~\dots~~~\strzd{(2,n-2)}\boxed{n-1}\strzd{(1,n-1)}
\boxed{n}\strzk{(0,n)}
\end{equation}
Let $${\bf n}=(n,n)\,,\quad\tt=\{t_1,t_2^{-1}\}\,,\quad \tt^{(a,b)}=t_1^{a}t_2^{-b}\,,\quad\hh=\{\muu_1,\muu_2\}\,,\quad \hh^{(a,b)}=\muu_1^{a}\muu_2^b\,.$$
Then
$$(*)=
\sum_{k\in V}\delta(\tt^{{\rm in}(k)},\hh^{{\bf n}-{\rm out}(k)})\cdot\delta(\tt^{-{\rm out}(k)},\hh^{{\bf n}-{\rm in}(k)})\,.$$
The second sum (**) is clearly the summation over the lattice points contained in a parallelogram
$$\Lambda=\big\{\tfrac kn-\tfrac \ell n \tau\in\C\;\big|\;k\in\{0,1,\dots,n-1\}\,,\;\ell\in\{0,1,\dots,n-1\}\big\}\,.$$
The set $\Lambda$ is in a bijection with the $n$-torsion points of $E=\C/\langle 1,\tau\rangle$. Note that in $(*)$ the variables of the same lower indices are mixed in arguments of $\delta$, while in $(**)$ they are separated.

\begin{example}\rm\label{A1} For $n=2$ we have
\begin{align*}(*)\quad&
\delta(t_1^2,\muu_1\muu_2)\delta(t_2/t_1,\muu_2^2)\;+
\delta(t_1/t_2,\muu_1^2)\delta(t_2^2,\muu_1\muu_2)\,\\
(**)\quad&\frac12 \Big(
\delta(t_1,\muu_1^2)\;\delta(t_2,\muu_2^2)\\
&\phantom{aaaa}+\delta(-t_1,\muu_1^2)\;\delta(-t_2,\muu_2^2)\\
&\phantom{aaaaaaa}+\delta(t_1q^{-1/2},\muu_1^2)\;\delta(t_2q^{1/2},\muu_2^2)\\
&\phantom{aaaaaaaaaa}+\delta(-t_1q^{-1/2},\muu_1^2)\;\delta(-t_2q^{1/2},\muu_2^2)\Big)\,.
\end{align*}
Note that for $t_1=t_2$ only $(**)$ makes sense since $\delta(x,y)$  has a pole at $x=1$.
In the additive notation  the formula reads
\begin{align*}(*)\quad
&\del(2t_1,\muu_1+\muu_2)\del(- t_1+t_2,2\muu_2)+
\del(t_1-t_2,2\muu_1)\del(2t_2,\muu_1+\muu_2)
\,,\\
(**)\quad&\frac1n\del(t_1+\tfrac kn -\tau\tfrac\ell n)
\cdot\del(t_2-\tfrac kn +\tau\tfrac\ell n)\Big)=\\&=\frac12 \Big(\del(t_1,2\muu_1)\del(t_2,2\muu_2)\\
&\phantom{aaaaaa}+\del(t_1+\tfrac12,2\muu_1)\del(t_2-\tfrac12,2\muu_2)\\
&\phantom{aaaaaaaaa}+\del(t_1-\tfrac12\tau,2\muu_1)\del(t_2+\tfrac12\tau,2\muu_2)\\
&\phantom{aaaaaaaaaaaa}+\del(t_1+\tfrac12-\tfrac12\tau,2\muu_1)\del(t_2-\tfrac12+\tfrac12\tau,2\muu_2)\Big)\,.
\end{align*}
\end{example}

Theorem \ref{A1main}  is a consequence of the equality of elliptic classes: according to \cite{BoLi} the class computed from the resolution and the orbifold elliptic class. In the next sections we will recall the necessary definitions and transform  them into a more convenient form, specific for symplectic singularities. The proof of Theorem \ref{A1main} is given in \S\ref{A1proof}.

\section{Elliptic class}
\subsection{Smooth variety}
Suppose that $X$ is a smooth complex variety. Then the elliptic class in cohomology is defined by the formula
$$\Ell_0(X)=\prod_{k=1}^{\dim X} x_k \frac{\theta(x_k-z)}{\theta(x_k)}\in H^*(X)[[q,z]]\,,$$
where $x_k$'s are the Chern roots of $TX$. In the multiplicative notation
$$\Ell_0(X)=\prod_{k=1}^{\dim X} x_k \frac{\vt(\xi_k\h)}{\vt(\xi_k)}\in H^*(X)[[q,z]]\,,$$
where $\xi_k=e^{x_k}$, $\h=\e(-z)$. (In the literature the indeterminate $\h$ is denoted by $h$ or by $\hbar$, which should rather be $e^{2\pi \ii\hbar}$. We want to keep the plain $h$ to denote elements of a finite group acting on $X$.)
The normalized elliptic class is defined by
\begin{equation}\label{normalization}\Ell(X)=\left(\tfrac{\vt'(1)}{\vt(\h)}\right)^{\dim X}\Ell_0(X)=eu(TX)\prod_{k=1}^{\dim X}\delta(e^{x_k},\h)\,.\end{equation}
Here $eu(TX)$ denotes the Euler class in $H^*(X;\Q)$.
The constant $\tfrac{\vt'(1)}{\vt(\h)}$ appearing in the definition of $\Delta$ is chosen to have $$\lim_{x\to 0}\left(x\delta(e^x,\h)\right)=1\,.$$
(We note that for the normalization used in \cite{BoLi0} the analogous limit is equal to $\frac1{2\pi \ii}$.)
The multiplicative notation has a deeper sense.  Instead of the elliptic class in cohomology we can consider \say{elliptic bundle}
 in the K-theory, see \cite[Formula (3)]{BoLi}.

\begin{example}\rm The elliptic nonequivariant genus is defined as the integral of the elliptic class. For the projective line it is equal to
$$Ell(\P^1):=\int_{\P^1}\Ell(\P^1)=\lim_{t\to 1}\left(\delta(t,\h)+\delta(t^{-1},\h)\right)=2\frac{\h \vt '(\h)}{\vt(\h)}\,.$$
\end{example}

\subsection{Interpretation of the elliptic class as the equivariant Euler class in elliptic cohomo\-lo\-gy}\label{explanation}
Naturally the elliptic class belongs to a version of equivariant elliptic cohomology $E_{\T\times \C^*}(X)$. We present below an explanation.
We assume that $E$ is an equivariant, complex oriented genera\-li\-zed cohomology theory, with a map $\Theta$ to equivariant cohomology extended by a formal parameter $q$, such that the Euler class in $E$ of a line bundle $L$  is mapped to the theta function:
$$\begin{matrix}\Theta:&E_\T(X)&\to&\muu_\T(X;\Q)\otimes \C((q))\\ \\
&eu^E_{\T}(L)&\mapsto &\tfrac{2\pi \ii}{\theta'(0)}\, \theta\!\left(\frac{c_1(L)}{2\pi \ii}\right)&=&\tfrac{1}{\vt'(1)}\, \vt\!\left(e^{c_1(L)}\right)\,.\end{matrix}$$
For the notion of the orientation and Euler class in generalized cohomology theories see \cite[Chapter 5]{Stong},  \cite[\S42]{FF}. A version of the theta function as a choice for the Euler class (or equivalently the choice of the related formal group law) appears in the literature, see \cite[p.197]{SegalEll} and in \cite{Totaro}. For our purposes it is enough to take as the equivariant elliptic cohomology the usual (completed) Borel equivariant cohomology
$E_\T(-)=\hat \muu_\T(-;\Q)\otimes \C((q))$ with the formal group law $F(x,y)=\theta(\theta^{-1}(x)+\theta^{-1}(y))$.
Suppose a torus $\T$ (possibly $\T$ is trivial) acts on $X$. We consider a bigger torus $\wTT=\T\times \C^*$, where $\C^*$ acts on $X$ trivially. The bundle $TX$ is an $\wTT$-equivariant bundle with the $\C^*$ action via the scalar multiplication. Formally we write $TX\otimes \h$, while $TX$ denotes the tangent bundle with the trivial action of $\C^*$. Let $$z=-\tfrac{c_1(\h)}{2\pi \ii }\in H^*_\wTT(\{{\rm pt}\})\,,\qquad \h=e^{-z}\,.$$
The elliptic class of $X$ is defined as the elliptic Euler class of the equivariant bundle $TX\otimes \h$. The elliptic genus of $X$ is defined as the push forward to the point 
of
$eu^E_{\wTT}(TX\otimes \h)$.
 By the generalized Riemann-Roch theorem \cite[42.1.D]{FF} the push-forward in $E$ can be replaced by the push forward  of the cohomology class
\begin{multline*}\tfrac{eu_\wTT(TX)}{\Theta(eu_\wTT^E)}\cdot\Theta\!\left(eu^E_\wTT(TX\otimes \h)\right)=
\tfrac{\prod_{k=1}^{\dim X}x_k}{\prod_{k=1}^{\dim X}\vt(e^{x_k})}\cdot\prod_{k=1}^{\dim X}\vt(e^{x_k} \h)=\\
=eu_\wTT(TX)\cdot\prod_{k=1}^{\dim X}\frac{\vt(e^{x_k} \h)}{\vt(e^{x_k})}\;\;\in\;\;  \muu_\wTT(X;\Q)\otimes
\C((q))
\,.\end{multline*}
Normalization \eqref{normalization} by the factor $\left(\frac{\vt'(1)}{\vt(\h)}\right)^{\dim X}$ might be interpreted as computing the \say{virtual} Euler class of the bundle $$TX\otimes \h-\h^{\oplus \dim X}\,.$$
The price we have to pay is that we invert $q$. The result belongs to $\hat H^*_\wTT(X;\Q)((q))$.
\medskip

It is natural to write the analogous transformation $\Theta^K$ to K-theory. Nevertheless the requirement $eu^E(L)\mapsto \vt(L)\in K_\T(-)\otimes\C((q))$ does not make sense, since in the definition of $\vt$ the square root of the argument appears. On the other hand the formula $$eu^K(L)\,\delta(L,\h)=(1-L^{-1})\,\delta(L,\h)$$ does make sense.
Therefore the (image of the) elliptic class is defined in the equivariant K-theory.
If the fixed point  $x\in X^\T$ is isolated, then the localized classes is given by the formula
$$\Theta^K\big(eu^E_\T(TX\otimes \h-\h^{\oplus \dim X})\big)_{|x}= \prod_{k=1}^{\dim X}\vt({\xi_k})\delta({\xi_k}, \h)\,,$$ where $\xi_k$'s are the Grothendieck roots of $TX$. The result is understood formally in a completion of the ring ${\rm R}(\T)\otimes \C((q))$.\medskip

The elliptic cohomology as a generalized cohomology theory was constructed in several setups, still the construction of an equivariant version does not seem to be  satisfactory. Of course rationally the theory is much easier. In the modern approach the elliptic cohomology ring $E_\T(X)$ is replaced by a scheme, and the Euler class is a section of some line bundle over that scheme, see \cite{GKV}, \cite[\S7.2]{Ga}, \cite[\S3.2]{ZZ}.

\subsection{Elliptic class of a singular variety admitting a crepant resolution}
The elliptic class of a singular variety was defined by Borisov and Libgober. In their original paper \cite{BoLi} they define a cohomology class for a suitable resolution of singularities $Y\to X$. The classes agree whenever one resolution is dominated by another. The push forward of that class  to $X$ defines a homology class of the singular variety itself, which does not depend on the choice of a resolution. The equivariant version of the theory works equally well: it is treated in \cite{Wae}, \cite{DBW}, \cite{RW}. We will assume that a torus $\T$ is acting on $X$ and $Y$ and the
resolution map $f:Y\to X$ is equivariant. For convenience suppose that $X$ is equivariantly embedded in a smooth ambient space $M$ and the set of fixed points $M^\T$ is finite. Let's assume that the fixed point set $Y^\T$ is finite as well. Then using Lefschetz-Riemann-Roch localization theorem for the push forward (see e.g.~\cite[Th.~5.11.7]{ChGi}, \cite[Prop.~2.7]{RW}) we can express the elliptic characteristic class as the sum of the terms $\frac{\Ell^\T(X)_{|x}}{eu(x,M)}$ depending on the local data coming from the fixed points $Y^\T$. If the resolution is crepant, then the localized elliptic class of $X$ restricted to a fixed point $x\in X^\T$ satisfies
\begin{equation}\label{rescrepformula}\frac{\Ell^\T(X)_{|x}}{eu(x,M)}=
\sum_{\tx\in f^{-1}(x) }\frac{\Ell^\T(Y)_{|\tx}}{eu(\tx,Y)}=
\sum_{\tx\in f^{-1}(x) }
\prod_{k=1}^{\dim Y} \delta(\tt^{w_i}(\tx),\h)
\,,
\end{equation}
where $w_1(\tx),w_2(\tx),\dots,w_{\dim Y}(\tx)$ are the weights of the torus action on  $T_\tx Y$ and $eu(x,M)$ is the product of the Chern roots of $T_xM$.

\begin{remark}\rm The original notation, given by Borisov and Libgober, is additive and uses an indeterminate $z$. Our $\h$ translates to $\e(-z)$. Also the equivariant variables $\tt=(t_1,t_2,\dots, t_n)$ should be understood as the exponents of the generators of $H^2_\T({{\rm pt}})$, or the elements of the representation ring ${\rm R}(\T)$.\end{remark}

\subsection{The relative elliptic class}
If the resolution $f:Y\to X$ is not crepant, then the formula \eqref{rescrepformula} for the elliptic class should be corrected by the discrepancy divisor. It is natural to consider the pairs $(X,D_X)$ from the beginning, where $D_X$ is a Weil $\Q$-divisor, such that $K_X+D_X$ is $\Q$-Cartier. Then
$\Ell^\T(X,D_X)$ is defined as $f_*\Ell^\T(Y,D_Y)$ whenever $f^*(K_X+D_X)=K_Y+D_Y$ and the pair $(Y,D_Y)$ is a resolution of $(X,D_X)$.
The formula for $\Ell^\T(Y,D_Y)$ makes sense only when the coefficients of $D_Y$ are not equal to 1, and to show  independence on the resolution Borisov and Libgober assume that the coefficients are smaller then 1 (equivalently, that the pair $(X,D_X)$ is Kawamata log-terminal). We give the formula for $\Ell^\T(Y,D_Y)$ in the equivariant case with isolated fixed point set. We assume that in some equivariant coordinates $$D_Y=\{ z_1^{a_1}\, z_2^{a_2}\dots  z_n^{a_n}=0\}$$ and the weight of the coordinate $z_k$ is equal to $w_k$ for $k=1,2,\dots, n=\dim Y$. Then
\begin{equation}\label{resolutionformula}\frac{\Ell^\T(Y,D_Y)_{|\tx}}{eu(\tx,Y)}=
\prod_{k=1}^n \delta(
\tt^{w_k},\h^{1-a_k})
\,.
\end{equation}
The symbol $\delta(
\tt^{w_k},\h^{1-a_k})$ is understood formally. It is a Laurent power series belonging to a suitable extension of the representation ring ${\rm R}(\T\times \C^*)={\rm R}(\T)[\h^{\pm1}]$, or via the Chern character to the completion of $H^*(B\T)[z]$, see \cite{DBW},\cite{RW} for further explanations.

\bigskip

If the fixed point set $Y^\T$ is not finite, as it happens for the singularity $D_4$ then the summation runs over the components of the fixed points. In that case one has to apply Atiyah-Bott-Berline-Vergne localization theorem in its full form.  In concrete cases one can  find a neighbourhood of the fixed component on which a bigger torus is acting with finitely many fixed points. In the case of $D_4$ it is possible to proceed that way.
\bigskip

We can extend the point of view presented in \S\ref{explanation}. Locally we compute the equivariant elliptic Euler characteristic of the bundle $TX$ twisted by $h$ in some power, not in the homogeneous manner, but the twisting depends on the direction. The exponents are encoded in the divisor $D_Y$. The relative elliptic class is the image of the Euler class 
of the virtual bundle
$$\bigoplus_{k=1}^{\dim Y} \xi_k\otimes \h^{1-a_k}\,- \,\bigoplus_{k=1}^{\dim Y} \h^{1-a_k}\,\,,$$
where $\xi_k$'s are the Grothendieck roots of $T_\tx Y$ adapted to the divisor $D_Y$.
To make sense of that formula we
extend the coefficients by the roots of line generators.
Equivalently, we replace the torus $\T$ by its finite cover.
 The localized elliptic classes considered in \cite{RW} belong to that ring.

\section{Resolution of singularities and theta identities}
Before discussing theta identities related to the quotient singularities let us give examples of identities which can be deduced from the invariance of the elliptic class with respect to the change of the resolution.

\subsection{Blowup and the Fay trisecant relation}\label{Fay}
  Let $X=\C^n$, $D_X=\sum_{i=1}^na_i D_i$, where $D_i=\{z_i=0\}$. Let $Y$ be the blowup of $\C^n$ at 0. The exceptional divisor is denoted by $E\simeq \P^{n-1}$. Then
$$D_Y=\sum_{i=1}^na_i\widetilde D_i+\left(\sum_{i=1}^n a_i-n+1\right)E$$
where
$\widetilde D_i$ is the strict transform of $D_i$. The torus $\T=(\C^*)^n$ acts coordinatewise on $X$.
There are $n$ fixed points of $\T$ acting on the blow-up $Y$, they belong to the exceptional divisor. At the point $[1:0:\dots:0]$ the characters of the action are the following:
$$t_1\;\;\text{in the normal direction, }\quad \tfrac{t_j}{t_1}\;\text{ for }j=2,\dots, n\;\;\text{in the tangent directions.}$$
At the remaining fixed points the formulas differ by a permutation of variables. We compute the local equivariant elliptic class applying the formula \eqref{resolutionformula}.
By Lefschetz-Riemann-Roch \cite[Th.~5.11.7]{ChGi} for the blow-down map $f:Y\to X$  we obtain the formula for the push forward
$$\frac{f_*(\Ell(Y))_{|0}}{eu(0,\C^n)}=
\sum_{i=1}^n\Big(\delta\big(t_i,{\textstyle\h^{n-\sum_{j=1}^n a_j}}\big)\cdot\prod_{j\neq i}\delta\big(\tfrac{t_j}{t_i},\h^{1-a_j}\big)\Big)\,.$$
This sum is equal to the expression computed directly on $X$, i.e. to the product of the $\delta$ functions. Setting $\muu_i=\h^{1-a_i}$ and substituting in the above formula we arrive to the following identity.
\begin{theorem} For $n\geq 2$ we have \begin{equation}\label{Fayhigh}\prod_{i=1}^n\delta\big(t_i,\muu_i\big)=\sum_{i=1}^n\Big(\delta\big(t_i,{\textstyle\prod_{j=1}^n \muu_j}\big)\cdot\prod_{j\neq i}\delta\big(\tfrac{t_j}{t_i},\muu_j\big)
\Big)\,.\end{equation}
\end{theorem}
The identity \eqref{Fayhigh} for $n=2$,
 is the {\it three term identity} given in \cite[Example 2.9]{RW}, which agrees with \cite[formula (2.6)]{RTV}.
When
written in the additive notation, after the substitution
$$t_1= a - d\,,\qquad t_2 = c - d\,,\qquad \mu_1 = c + d\,,\qquad \mu_2 = b - c$$
and multiplication by the common denominator
the identity takes the form
of the  Fay's trisecant identity (see \cite{GTL}, \cite[\S5.2]{Risu}).
\begin{multline}
 \theta(a+c)\theta(a-c)\theta(b+d)\theta(b-d)=\\
 \theta(a+b)\theta(a-b)\theta(c+d)\theta(c-d)+\\+ \theta(a+d)\theta(a-d)\theta(b+c)\theta(b-c)\,.
\label{trisecant}\end{multline}
For arbitrary $n$ the formula  \eqref{Fayhigh} can be transformed to a symmetric form in the following way:
Set
\begin{align*}
t_i=&\frac{x_i}{x_0}\,,\qquad\mu_i=\frac{\xi_i}{\xi_{i-1}}\qquad\text{ for } i=1,2\dots n\,.\\
\end{align*}
Then the first factor of the RHS of \eqref{Fayhigh} is equal to
$$\delta\left(t_i,\textstyle{\prod_{j=1}^n\mu_j}\right)=\delta\left(\frac{x_i}{x_0},\frac{\xi_n}{\xi_0}\right)=-\delta\left(\frac{x_0}{x_i},\frac{\xi_0}{\xi_n}\right)\,.$$
The identity \eqref{Fayhigh} can be rewritten as
\begin{equation}\sum_{i=0}^n\;\prod_{\begin{matrix}j=0,1,\dots, n\\j\neq i\end{matrix}}\delta\left(\frac{x_j}{x_i},\frac{\xi_j}{\xi_{j-1}}\right)=0\,.\end{equation}
Here $\xi_{-1}=\xi_n$.


\subsection{Braid relation}
The following expressions are equal:
\begin{equation}\label{s1s2s1}
\delta\left(\frac{t_2}{t_1},\frac{\muu_3}{\muu_2}\right)
\delta\left(\frac{t_3}{t_2},\frac{\muu_3}{\muu_1}\right)
\delta\left(\frac{t_2}{t_1},\frac{\muu_2}{\muu_1}\right)
+
\delta\left(\frac{t_1}{t_2},\h\right)
\delta\left(\frac{t_3}{t_1},\frac{\muu_3}{\muu_1}\right)
\delta\left(\frac{t_2}{t_1},\h\right),
\end{equation}
\begin{equation}\label{s2s1s2}
\delta\left(\frac{t_3}{t_2},\frac{\muu_2}{\muu_1}\right)
\delta\left(\frac{t_2}{t_1},\frac{\muu_3}{\muu_1}\right)
\delta\left(\frac{t_3}{t_2},\frac{\muu_3}{\muu_2}\right)
+
\delta\left(\frac{t_2}{t_3},\h\right)
\delta\left(\frac{t_3}{t_1},\frac{\muu_3}{\muu_1}\right)
\delta\left(\frac{t_3}{t_2},\h\right),
\end{equation}
\cite[\S9]{RW}.
Geometrically the above expressions come from two natural resolutions of the singularity $$z_{31}(z_{31}-z_{21}z_{32})=0$$
describing the boundary of the big cell in the complete flag variety $GL_3(\C)/B$ intersected with the opposite cell, see \cite[Ex.~16.1]{WeSel}. The equality
is equivalent to the {\em four term identity} \cite[eq.~(2.7)]{RTV}.

Strangely, the Braid relation for the group $Sp_2(\C)$ is trivial and for $G_2$ it can be deduced from $SL_3(\C)$.

\subsection{Lehn--Sorger example}\label{Lehn-Sor} Let $G\to Sp_2(\C)\subset GL_4(\C)$ be the example described in \cite{LeSo}, \cite{Grab}. Here $G$ is the bi-tetrahedral group. The linear space $V$ is the direct sum  $W\oplus W^*$, where $W$ is one of  two-dimensional  non-self-conjugate representations of $G$. The quotient $(W\oplus W^*)/G$ admits two crepant resolutions.
The computation of the tangent weights can be found in \cite[Th.~4.5]{Grab}.
The equivariant elliptic class computed from these resolutions give the same result, but the expressions for the elliptic class differ by the switch of variables.
After subtraction of the identical summands on both sides, the resulting identity can be written as:
$$F(t_1,t_2)=F(t_2,t_1)\,,$$
where
\begin{multline}\label{F-def}F(t_1,t_2)=\delta\left(\tfrac{t_1}{t_2},\h\right)\delta\left(\tfrac{t_1^3}{t_2},\h\right)\delta\left(t_2^2,\h\right)\delta\left(\tfrac{t_2^2}{t_1^2},\h\right)+\\+\delta\left(t_1^2,\h\right)\delta\left(\tfrac{t_1^4}{t_2^2},\h\right)\delta\left(\tfrac{t_2}{t_1},\h\right)\delta\left(\tfrac{t_2^3}{t_1^3},\h\right)+\\+\delta\left(\tfrac{t_1^3}{t_2^3},\h\right)\delta\left(\tfrac{t_1^2}{t_2^2},\h\right)\delta\left(\tfrac{t_2^3}{t_1},\h\right)\delta\left(\tfrac{t_2^4}{t_1^2},\h\right)\,.\end{multline}
Our formula is nothing but plugging in the exponents (of $t$-variables) given by \cite[Th.~4.5]{Grab}.
This example is continued in \S\ref{LScont}.
\section{The orbifold elliptic class}
The orbifold elliptic class is defined in the presence  of an action of a finite group $G$. Again, for a singular equivariant pair it is defined as the image of the orbifold elliptic  class of a resolution:
$$\Ell_{orb}(X,D_X,G)=f_*\Ell_{orb}(Y,D_Y,G)\,.$$
Let us quote the original definition of \cite[Def.~3.2]{BoLi} in a precise form.
Let $(Y,D_Y)$ be a Kawamata log-terminal $G$-normal pair (cf \cite[Def.~3.1]{BoLi})
with $D_Y=-\sum_\ell d_\ell E_\ell$.
We define \emph{orbifold elliptic class} of the triple $(Y,D_Y,G)$ as
an element of  $H^*(Y)$ (or the Chow group) by the formula
$$\Ell_{orb}(Y,D_Y,G):=
\frac 1{\vert G\vert }\sum_{g,h,gh=hg}\sum_{Z\subset Y^{g,h}}(i_{Y^Z})_*
\Bigl(
\prod_{k:\;\lambda^Z_k=\nuu^Z_k=0} x_k
\Bigr)
$$
$$\times\prod_{k} \frac{ \theta(\frac{x_k}{2 \pi \ii }+
 \lambda^Z_k-\tau \nuu^Z_k-z )}
{ \theta(\frac{x_k}{2 \pi \ii }+
 \lambda^Z_k-\tau \nuu^Z_k)}  e^{2 \pi \ii \nuu^Z_kz}
$$
$$\times\prod_{\ell}
\frac
{\theta(\frac {e_\ell}{2\pi\ii}+\epsilon_\ell^Z-\zeta_\ell^Z\tau-(d_\ell+1)z)}
{\theta(\frac {e_\ell}{2\pi\ii}+\epsilon_\ell^Z-\zeta_\ell^Z\tau-z)}
\,\frac{\theta(-z)}{\theta(-(d_\ell+1)z)} e^{2\pi\ii d_\ell\zeta_\ell^Zz}.
$$
Here $Z\subset Y^{g,h}$
is an irreducible component of the fixed set of the commuting
elements $g$ and $h$ and $i_{Z} : Z\to Y$ is the corresponding embedding. The
restriction of $TY$ to $Z$ splits as a sum of one dimensional representations on which $g$ (respectively $h$) acts with the eigenvalues $\e(\lambda^Z_k)$ (resp. $\e(\nuu^Z_k)$),  $\lambda^Z_k,\nuu^Z_k\in \Q \cap [0, 1)$. The Chern roots of $(TY)_{|Z}$ are denoted by $x_k$. In addition, $e_\ell = c_1 (E_\ell )$ and $\e(\epsilon_\ell^Z)$, $\e(\zeta_\ell^Z)$ with $\epsilon_\ell^Z,\zeta_\ell^Z \in \Q \cap [0, 1)$
are the eigenvalues of $g$ and $h$ acting on ${\mathcal O}(E_\ell)$ restricted to $Z$ if $E_\ell$ contains $Z$ and is zero otherwise.
\medskip

Let us assume that the torus action commutes with the action of $G$. The formula for the $\T$-equivariant orbifold elliptic class  simplifies significantly when the fixed point set is discrete. For a normal crossing divisor situation with isolated fixed points
 we can assume that (at the fixed points) the divisor classes coincide with the Chern roots. Then, in the multiplicative notation and after the correction by the factor $(2\pi \ii)^{\dim Y}$, the local formula for the orbifold class takes the form
\begin{multline}\label{elldefor}\frac{\Ell_{orb}^\T(Y,D_Y,G)}{eu(\tx,Y)}=
\frac1{|G|}\sum_{g,h,gh=hg}
\prod_{k=1}^n \delta\left(\,\e(\lambda^{g,h}_k- \nuu^{g,h}_k\tau)\tt^{w_k}\,,\,\h^{1-a_k}\,\right)\,\h^{(a_k-1)\nuu^{g,h}_k}
\end{multline}
\begin{equation}\label{mckformula}
=\frac1{|G|}\sum_{g,h,gh=hg}
\prod_{k=1}^n \delta\left(\,
\zeta^{g,h}_k q^{-\nuu^{g,h}_k}\tt^{w_k}\,,\,\h^{1-a_k}\,\right)\,\h^{(a_k-1)\nuu^{g,h}_k}
\,,
\end{equation}
where $
\zeta^{g,h}_k=\e(\lambda^{g,h}_k)$ for $k=1,2,\dots,\dim Y$ are the eigenvalues of $g$ acting on $(TY)_{|Y^{g,h}}$. The fixed point sets $Y^{g,h}$ are considered locally, therefore we do not use $Z$ as the superscript of eigenvalues, but the pair $(g,h)$.
The main result of \cite{BoLi} states that
\begin{theorem}[\cite{BoLi}, Theorem 5.3]\label{BoLimK}
Let $(X;D_X)$ be a Kawamata log-terminal pair which is invariant under
an effective action of $G$ on $X$. Let $\psi: X\to X/G $
be the quotient morphism.  Then
$$
\psi_* \Ell_{orb}(X,D_X,G)=\Ell(X/G,D_{X/G}),
$$
provided that $\psi^*(K_{X/G} + D_{X/G}) = K_X + D_X$.
\end{theorem}
For the equivariant version see \cite{Wae}, \cite{DBW}.
If $X$ is a vector space with a linear action of a finite group $G$ commuting with $\T$, then the $\T$-equivariant version of the equality above is equivalent to the equality of Laurent power series
$$\frac{\Ell_{orb}^\T(X,D_X,G)_{|0}}{eu(0,X)}=\frac{\Ell^\T(X/G,D_{X/G})_{|\psi(0)}}{eu(\psi(0),M)}\,.$$
If $D_{X/G}=0$ and  $Y\to X/G$ is a resolution then the right hand side is given by \eqref{mckformula} and the left hand side is the sum over the fixed points $Y^\T$ of the expressions given in the formula \eqref{resolutionformula}.  If the resolution is crepant, then the formula becomes simpler, the summands are of the form \eqref{rescrepformula}.
\bigskip

Interpretation of the orbifold elliptic class in the spirit of \S\ref{explanation} is somehow ambiguous. The conjugate elements $h\in G$ give the same contribution to the sum \eqref{elldefor}. Therefore this sum can be reorganized, so that we sum over the conjugacy classes $[h]\in Conj(G)$ and $g$ belongs to the centralizer $C(h)$:
 \begin{multline}\label{inertia}\frac{\Ell_{orb}^\T(Y,D_Y,G)}{eu(\tx,Y)}=
\sum_{[h]\in Conj(G)}\frac1{|C(h)|}\sum_{g\in C(h)}
\prod_{k=1}^n \delta\left(\,
\zeta^{g,h}_k q^{-\nuu^{g,h}_k}\tt^{w_k}\,,\,\h^{1-a_k}\,\right)\,\h^{(a_k-1)\nuu^{g,h}_k}
\,.
\end{multline}
In the limit with $q\to 1$ we obtain the formula \cite[Th.~13, Cor.~14]{DBW}, which can be interpreted as the summation over components of the \emph{extended quotient} 
$$\bigsqcup_{[h]\in Conj(G)}X^h/C(h)\,.$$
For the elliptic orbifold class this interpretation is only partial: the formula depends on the normal bundle of $X^h$ in $X$. The normal factor becomes trivial only in the limit, due to \cite[equation (21)]{DBW}.
The summand corresponding to $h=\rm id$ (the identity) is equal to
$$\frac1{|G|}\sum_{g\in G}
\prod_{k=1}^n \delta
\left(\,\zeta^{g}_k \tt^{w_k}\,,\,\h^{1-a_k}\,\right)\,,
$$
where $\zeta^g_k=\zeta^{g,\rm id}_k$ for $k=1,\dots n$ are the eigenvalues of $g$ acting on $T_\tx Y$.
The formula can be treated as the averaged elliptic class of $(Y,D_Y)$. If $Y=V$ is a vector space, $a_k=0$, $\tt^{w_k}=t$, then the formula is a deformation of the expression for the classical Molien series, see \cite[\S9]{DBW}.
Precisely, it is the weighted dimension of the invariants of the \say{elliptic representation} \medskip

\hfil $\left(\h^{\dim V/2}\, \bigotimes_{n \ge 1}
\Bigl(\bigwedge_{-q^{n-1}\h^{-1}}V^* \otimes \bigwedge_{- q^n\h}
V \otimes S_{q^{n-1}}V^* \otimes
S_{q^n}V \Bigr)\right)^G\,.$\medskip

\noindent This elliptic representation only differs from \cite[Formula (3)]{BoLi} by the factor
$S_1V^*=Sym(V^*)$ playing the role of the inverse of K-theoretic Euler class. The remaining summands of the formula \eqref{inertia} for $h\neq \rm id$ are more complicated.

\section{Orbifold elliptic class of  symplectic singularities}\label{genus}
Let us concentrate on the case of symplectic quotient singularities.
Let $V= \C^{2n}$ be endowed with the standard symplectic structure and let $G\subset Sp_n(\C)
$ be a finite subgroup.
Then the eigenvalues $\lambda^{g,h}_k$ and $\nuu^{g,h}_k$ come in pairs.
We can assume that
$$\nuu^{g,h}_{k+n}=\begin{cases}0&\text{if }\nuu^{g,h}_{k}=0\\
1-\nuu^{g,h}_{k}&\text{if }\nuu^{g,h}_{k}>0\,.\end{cases}$$
and similarly for $\lambda^{g,h}_k$. For $\nuu^{g,h}_{k}>0$ the pair of factors can be transformed:
\begin{multline}
\delta\left(\zeta^{g,h}_k q^{-\nuu^{g,h}_k} \tt^{w_k},\h^{1-a_k}\right)\h^{(a_{k}-1)\nuu^{g,h}_{k}}\cdot
\delta\left(\zeta^{g,h}_{n+k}q^{-\nuu^{g,h}_{n+k}} \tt^{w_{n+k}},\h^{1-a_{n+k}}\right)
\h^{(a_{n+k}-1)\nuu^{g,h}_{n+k}}=\\ \hfill=
\delta\left(\zeta^{g,h}_k q^{-\nuu^{g,h}_k} \tt^{w_k},\h^{1-a_k}\right)\cdot
\delta\left((\zeta^{g,h}_{k})^{-1}q^{\nuu^{g,h}_{n+k}-1} \tt^{w_{n+k}},\h^{1-a_{n+k}}\right)\cdot
\h^{(a_{k}-1)\nuu^{g,h}_{k}+(a_{n+k}-1)(1-\nuu^{g,h}_{k})}\,.\end{multline}
By \eqref{deltatau} we obtain
\begin{multline}\label{sympmckay}
\delta\left(\zeta^{g,h}_k q^{-\nuu^{g,h}_k} \tt^{w_k},\h^{1-a_k}\right)\cdot\delta\left((\zeta^{g,h}_{k})^{-1}q^{\nuu^{g,h}_{k}} \tt^{w_{n+k}},\h^{1-a_{n+k}}\right)\cdot\hfill\\ \hfill\cdot \h^{1-a_{n+k}}\cdot
\h^{(a_{k}-1)\nuu^{g,h}_{k}+(a_{n+k}-1)(1-\nuu^{g,h}_{k})} =\\ \hfill
=
\delta\left(\zeta^{g,h}_k q^{-\nuu^{g,h}_k} \tt^{w_k},\h^{1-a_k}\right)\cdot\delta\left((\zeta^{g,h}_{k})^{-1}q^{\nuu^{g,h}_{k}} \tt^{w_{n+k}},\h^{1-a_{n+k}}\right)\cdot
\h^{\nuu^{g,h}_{k}(a_k-a_{n+k})} \,.
\end{multline}
The same holds for $\nuu^{g,h}_{k}=0$.

If the torus acts via the scalar multiplication and the divisor is empty (i.e.~$a_k=0$), then the formula \eqref{sympmckay} reduces to
$$\delta\left(\zeta^{g,h}_k q^{-\nuu^{g,h}_k} t,\h\right)\cdot\delta\left((\zeta^{g,h}_{k})^{-1}q^{\nuu^{g,h}_{k}} t,\h\right) $$ or equivalently setting $\lambda=\lambda^{g,h}_k-\nuu^{g,h}_k\tau$ we obtain the factor
\begin{equation}\label{Phi-def}\Phi(\lambda):=\delta\left(\e(\lambda) t,\h\right)\cdot\delta\left(\e(-\lambda) t,\h\right) \end{equation}

If the divisor is empty, and  $\T=(\C^*)^2$ acts via the scalar multiplication on each summand in $V=W\oplus W^*$ separately, then we obtain the factor
\begin{equation}\label{Psi-def}\Psi(\lambda):=\delta\left(\e(\lambda) t_1,\h\right)\cdot\delta\left(\e(-\lambda) t_2,\h\right) \,.\end{equation}
The elliptic class of this kind of singularity has a symmetry property:
\begin{remark}\rm Suppose $V=W\oplus W^*$,  $G\subset GL(W)$, and $\T=(\C^*)^2$ acts via the scalar multiplication on each summand, as in  the example \S\ref{Lehn-Sor}. Then
the elliptic class of $V/G$ is a symmetric function with respect to the coordinate characters $t_1,t_2$. Indeed Since $\delta(a,b)=-\delta(a^{-1},b^{-1})$ the factor \eqref{Psi-def} in the orbifold elliptic class has the property
\begin{equation}\label{odwracanie}\Psi(\lambda)(t_1,t_2,\h)=\Psi(\lambda)(t_2^{-1},t_1^{-1},\h^{-1})\,.\end{equation}
Therefore
$$\frac{\Ell^\T_{orb}(V,G,\emptyset)}{eu(0,V)}(t_1,t_2,\h)=\frac{\Ell^\T_{orb}(V,G,\emptyset)}{eu(0,V)}(t_2^{-1},t_1^{-1},\h^{-1})\,
.$$
In the expression \eqref{resolutionformula} for the elliptic class coming from a resolution there is no shift of variables therefore
$$\frac{\Ell^\T(V/G,\emptyset)}{eu([0],M)}(t_2^{-1},t_1^{-1},\h^{-1})=\frac{\Ell^\T(V/G,\emptyset)}{eu([0],M)}(t_2,t_1,\h)\,
.$$
By Theorem \ref{BoLimK} we obtain the conclusion \eqref{odwracanie}.
The symmetry may be easily deduced geometrically.
\end{remark}

\section{Examples}
\subsection{The singularity $D_4$}

The singularity $D_4$ is the quotient of $\C^2$ by the bi-dihedral group of 8 elements which is isomorphic to the quaternionic group  generated by the matrices
$${\bf i}=\left(\begin{matrix}\ii& 0\\0& -\ii\end{matrix}\right)\,,\quad{\bf j}=
\left(\begin{matrix}0& 1\\-1& 0\end{matrix}\right)\quad\text{and}\quad {\bf k}={\bf i\,j}=
\left(\begin{matrix}0& \ii\\\ii& 0\end{matrix}\right)\,.$$ We consider the one dimensional torus $\T=\C^*$ acting via the scalar multiplication.

\subsubsection{The elliptic class computed via resolution}
The following quiver represents the resolution of the singularity $D_4$
$$\begin{matrix}\nwarrow \\
^4\phantom{a}&\bullet&\\
&&\!\!_2\!\!\nwarrow&\\
&&&\boxed{\P^1}&\stackrel{2}{\longrightarrow}&\bullet&\stackrel{4}{\longrightarrow}\\
&&\!\!^2\!\!\swarrow\\
_4\phantom{a}&\!\!\bullet&\\
\swarrow
\end{matrix}$$
\medskip

\noindent The internal arrows represent the exceptional divisors with nontrivial torus action.
The divisor $\boxed{\P^1}$ is fixed pointwise by $\C^*$.
The external arrows represent the normal directions pointing out from the exceptional divisor at the isolated fixed points. The number at each arrow stands for weight of the action along the divisor. The local equivariant elliptic class is given by the formula
\begin{equation}\label{*d4}\frac{\Ell(X)_{|[0]}}{eu([0])}=3\delta(t^{-2},\h)\cdot\delta(t^4,\h)+\int_{\P^1}\frac{\Ell^\T(Y)_{|\P^1}}{eu(N_{\P^1})}\,.\end{equation}
Here the integral is the localized elliptic genus integrated along the fixed component, $N_{\P^1}$ denotes the normal bundle.  It can be computed as in Example \ref{A1} artificially extending the torus: There exists a neighbourhood of the fixed component which admits a two dimensional torus action having only two fixed points. This neighbourhood is isomorphic to the neighbourhood of the exceptional divisor for the singularity $A_1$ therefore the integral  is the specialization of the sum (**) of Example \ref{A1}:
\begin{equation}\label{calka}\int_{\P^1}\frac{\Ell^\T(Y)_{|\P^1}}{eu(N_{\P^1})}= \tfrac12\left(\Phi(0)+
\Phi(\tfrac 12 )+
\Phi(-\tfrac1 2\tau)+
\Phi(\tfrac 12 -\tau\tfrac1 2)
\right)
\,,\end{equation}
where $\Phi(\lambda)$ is defined by \eqref{Phi-def}.

\subsubsection{The computation of the orbifold elliptic class}
The sub-sum of \eqref{mckformula} indexed by the pairs $(g,1)$ is equal to
$$S_{1}=\Phi(0)+\Phi(-\tfrac12)+6\Phi(-\tfrac14)\,.$$
The sub-sum of \eqref{mckformula} indexed by the pairs $(g,-1)$ is equal to
$$S_{-1}=\Phi(-\tfrac12\tau)+\Phi(\tfrac12-\tfrac12\tau)+6\Phi(\tfrac14-\tfrac12\tau)\,.$$
The remaining six possibilities $h\in\{\pm {\bf i},\pm {\bf j},\pm {\bf k}\}$ give
$$S_{{\bf i}}=\Phi(-\tfrac14\tau)+
\Phi(\tfrac14-\tfrac14\tau)
+\Phi(\tfrac12-\tfrac14\tau)
+\Phi(\tfrac34-\tfrac14\tau)\,.
$$
Therefore
$$\frac{\Ell^\T_{orb}(\C^2,\emptyset,G)_{|0}}{eu(0,\C^2)}=\tfrac18\left(S_{1}+S_{-1}+6S_{{\bf i}}\right)\,.$$
After simplification we obtain
\begin{multline*}\tfrac68\left(\Phi(\tfrac14)+
\Phi(\tfrac14-\tfrac14\tau)
+\Phi(\tfrac14-\tfrac12\tau)
+\Phi(\tfrac14-\tfrac34\tau)
+\Phi(-\tfrac12\tau)
+\Phi(\tfrac12-\tfrac12\tau)\right)-\\
-\tfrac38
\left(\Phi(0)
+
\Phi(\tfrac 12 )+
\Phi(-\tfrac1 2\tau)+
\Phi(\tfrac 12 -\tau\tfrac1 2)
\right)=3\delta(t^{-2},\h)\cdot\delta(t^4,\h)\,.\end{multline*}
The formula can be further transformed:
Since $\Phi(\lambda)=\Phi(1+\tau-\lambda)$ we can write
$6\Phi(\tfrac14-\tfrac14\tau)=3\Phi(\tfrac34-\tfrac34\tau)+3\Phi(\tfrac14-\tfrac14\tau)$. We break in two other terms with coefficients $6$ and obtain (dividing by 3) a remarkable formula
$$\delta(t^{-2},\h)\cdot\delta(t^4,\h)=\tfrac18\sum_{k,\ell=0}^3
(-1)^{(k+1)(\ell+1)}\,\Phi(\tfrac k4-\tfrac\ell4\tau)\,,$$
where $\Phi(\lambda)$ is given by \eqref{Phi-def}. This is a nontrivial relation involving theta function.

\subsection{Lehn--Sorger example continued} We compute the orbifold elliptic class for the example considered in \S\ref{Lehn-Sor}.
\label{LScont}
The Lehn-Sorger group $G\subset GL_2(\C)=GL(W)$ is generated by
$$h_1=
\tfrac{-1+\ii
   \sqrt{3}}{2}
\left(
\begin{array}{cr}
 \frac{1+\ii}{2} & -\frac{1-\ii}{2} \\
 \frac{1+\ii}{2} & \frac{1-\ii}{2} \\
\end{array}
\right)
\,,\qquad h_2=\left(\begin{array}{cr}\ii&0\\0&-\ii\\\end{array}\right)
\,.$$
It is important to know that
$$h_1^3=h_2^2=-id\in Z(G)\,,\quad C(h_1)=C(h_1^2)=\langle h_1\rangle\simeq\Z_6 \,,\quad C(h_2)=\langle h_2\rangle\simeq\Z_4\,.$$
The conjugacy classes of $h\in G$ with the logarithms of the eigenvalues  are listed below.
$$\def\arraystretch{1.5}\begin{array}{|c|c|c|c|}
\hline
h&\,|C(h)|\,&(\nuu_1,\nuu_2)&\text{logarithms of the eigenvalues of }g\in C(h):\;\;(\lambda_1,\lambda_2)\\
\hline\hline
id&24&(0,0)&
(0,0),\;(\tfrac12,\tfrac12),\;4\times(\tfrac16,\tfrac12),\;4\times(\tfrac13,0),\;4\times(\tfrac23,0),
\\
-id&24&(\tfrac12,\tfrac12)& \;4\times(\tfrac56,\tfrac12),\;6\times(\tfrac14,\tfrac34)
\\
\hline
h_1&6&(\tfrac16,\tfrac12)&\\
h_1^2&6&(\tfrac13,0)&(0,0),\;(\tfrac16,\tfrac12),\;(\tfrac13,0),\;(\tfrac12,\tfrac12),\;(\tfrac23,0),\;(\tfrac56,\tfrac12)
\\
h_1^4&6&(\tfrac23,0)&
\\
h_1^5&6&(\tfrac56,\tfrac12)&
\\\hline
h_2&4&(\tfrac14,\tfrac34)&(0,0),\;(\tfrac14,\tfrac34),\;(\tfrac12,\tfrac12),\;(\tfrac34,\tfrac14)
\\
\hline
\end{array}$$
Each quadruple $(\lambda_1,\lambda_2.\nuu_1,\nuu_2)$ contributes the summand
$$\Psi(\lambda_1+\tau\nuu_1)\cdot\Psi(\lambda_2+\tau\nuu_2)$$
to $\Ell_{orb}^\T(\C^4,G,\emptyset)/eu(0,\C^4)$. It is
counted with the weight $1/C(h)$. The formula for $\Psi$ is given in \eqref{Psi-def}.
\medskip

On the other hand, applying the computation of the tangent weights of the resolution presented in \cite[Th.~4.5]{Grab} we find that
$$\frac{\Ell^\T(\C^4/G,\emptyset)}{eu([0],M)}=F_0(t_1,t_2)+F_0(t_2,t_1)+F(t_1,t_2)\,,$$
where
\begin{multline*}F_0(t_1,t_2)=\\=
\delta \left(\tfrac{t_1}{t_2^5},\h\right)
\delta \left(\tfrac{t_1}{t_2^3},\h\right)
\delta \left({t_2^4},\h\right)
\delta \left({t_2^6},\h\right)+
\delta \left(\tfrac{t_1^2}{t_2^4},\h\right)
\delta \left(\tfrac{t_1}{t_2},\h\right)
\delta \left(t_2^2,\h\right)
\delta \left(\tfrac{t_2^5}{t_1},\h\right)\end{multline*}
and $F(t_1,t_2)$ is given by \eqref{F-def}. The equality of elliptic genera implies an identity for theta functions.
The explicit expanded form is too long to present it here.

\section{Diagonal quotient}
\subsection{The quotient $\C^m/\Z_n$}
Let $\Z_n$ act on $\C^m$ via the scalar multiplication by the $n$-th root of unity.
The quotient $X=\C^m/\Z_n$ has an isolated singularity and admits a desingularization via blow up at the origin. The resolution $Y$ is isomorphic to the total space of the bundle $\O(-n)$ over $\P^{m-1}$. The torus $\T=(\C^*)^m$ acts on $\C^m$ coordinatewise and the action commutes with $\Z_n$. Setting $D_X=0$ we find that $D_Y=\pi^*K_X-K_Y$ is
supported by the exceptional divisor $\P^{m-1}$
$$D_Y\;=\;(1-\tfrac mn)\P^{m-1}\,.$$
Therefore the localized equivariant elliptic class is equal to
\begin{equation}\frac{\Ell(X)_{[0]}}{eu([0],M)}=\int_{\P^{m-1}} \delta(e^{c_1^\T(\O(-n))},\h^{m/n})\Ell(\P^{m-1})\,.\end{equation}
By the localization theorem for the full torus we obtain
\begin{equation}\label{LGformula}\frac{\Ell(X)_{[0]}}{eu([0],M)}=\frac1n\sum_{i=1}^m\left(\delta\left(t_i^n,\h^{m/n}\right) \prod_{j\neq i}\delta\left(\tfrac{t_j}{t_i},\h\right)\right)\,.\end{equation}
The orbifold elliptic class is equal to
\begin{equation}\label{MKformula}\frac{\Ell_{orb}(\C^m,\emptyset,\Z_n)_{0}}{eu(0,\C^m)}=\sum_{k,\ell=0}^{n-1}\left(\h^{-\tfrac{m\ell}n}\prod_{i=1}^m\delta\left(\e\left(\tfrac kn-\tfrac \ell n\tau\right) t_i,\h\right)\right)\,.\end{equation}
Again \eqref{LGformula}=\eqref{MKformula} is a nontrivial identity for the theta function.
This equality  for $m=1$ is mentioned in \cite[Cor. 8.2]{BoLi}.

\subsection{Mixing variable types}\label{LaGi}
Assume that $m=n$. It was observed in \cite{LibLG} while analyzing the Landau-Ginzburg model, that if we restrict the action to the one dimensional torus and set the equivariant variable $t=\h^{-1/n}$ (note that $\C^*/\Z_n$ acts effectively on $X=\C^n/\Z_n$) then we obtain the formula for the Elliptic genus of the Calabi-Yau hypersurface in $\mathbb{H}_{CY}\subset\P^{n-1}$. Let us transform this calculation to our notation.
In the K-theory $[T\P^{n-1}]=[n\O(1)]-[\O]$, hence $$\Ell(\P^{n-1})=x^n\,\delta(e^x,\h)^n\,,$$
where $x=c_1(\O(1))$. Let $$u=ch(\O(1))=e^{x}\in H^*(\P^{n-1})$$
be the exponent of the nonequivariant Chern class. Then
\begin{multline*}\frac{\Ell(X)_{[0]}}{eu([0],M)}=\int_{\P^{n-1}}\frac{\Ell^\T(Y)_{|\P^{n-1}}}{eu(N_{\P^{n-1}}))}=\\=\int_{\P^{m-1}} x^n\delta((t/u)^n,\h)\delta(u,\h)^n=\int_{\P^{m-1}} x^n\tfrac{\vt((t/u)^n\h)\vt'(1)}{\vt((t/u)^n)\vt(\h)}
\delta(u,\h)^n=\\
\stackrel{t:=\h^{-1/n}}=\;\int_{\P^{m-1}} x^n
\tfrac {\vt(u^{-n})\vt'(1)}
{\vt(u^{-n}\h^{-1})\vt(\h)}
\delta(u,\h)^n=\int_{\P^{m-1}}x^n
\tfrac {\vt(u^{n})\vt'(1)}
{\vt(u^{n}\h)\vt(\h)}
\delta(u,\h)^n=\\
=\left(\tfrac {\vt'(1)}
{\vt(\h)}\right)^2\int_{\P^{m-1}} x^n
\delta(u^{n},\h)^{-1}
\delta(u,\h)^n=\left(\tfrac {\vt'(1)}
{\vt(\h)}\right)^2Ell(\mathbb{H}_{CY})\,.
\end{multline*}
When we consider unreduced elliptic genera then we get rid of the factor $\left(\tfrac {\vt'(1)}
{\vt(\h)}\right)^2$.
It is interesting to observe that the integral can be expressed by a residue:
\begin{multline*}\int_{\P^{n-1}}\frac{\Ell^\T(Y)_{|\P^{n-1}}}{eu(N_{\P^{n-1}})}=\int_{\P^{n-1}}x^n\delta(t/e^x,\h)\delta(e^x,\h)^n=\\
=\text{Coefficient of }x^{n-1}\text{ in } x^n \delta(t/e^x,\h)\delta(e^x,\h)^n={\rm Res}_{x=0}\big(\delta(t/e^x,\h)\delta(e^x,\h)^n\big)\\
={\rm Res}_{u=1}\big(\delta(t/u,\h)\delta(u,\h)^n/u\big)\,.\end{multline*}

\section{Proof of Theorem \ref{A1main}}

\label{A1proof}
Let $\T=(\C^*)^2$ be the torus acting on $\C^2$ coordinatewise. Denote the coordinates on $\C^2$ by $z_1, z_2$. Let $\Z_n\subset SL_2(\C)\cap \T$ be the subgroup generated by ${\rm diag}(\e(\tfrac 1n),\e(-\tfrac {1}n))$. The action of $\T$ passes to the quotient $X=\C^2/\Z_n$.
The minimal resolution of $X$ is a toric variety having $n$ fixed points joined by the chain of one dimensional orbits of $\T$. The GKM graph of $Y$ is given in \eqref{GKM}. The external edges have loose ends since $Y$ is not compact. They correspond to the strict transforms of the divisors coming from $X$, namely $D_1=\{z_1=0\}/\Z_n$ and $D_2=\{z_2=0\}/\Z_n$. Suppose $D_X= a_1 D_1+a_2 D_2$ is the divisor given by the function $f=z_1^{a_1}z_2^{a_2}$. The divisor $D_Y$ is supported by the strict transforms and $\widetilde D_1$, $\widetilde D_2$ and the exceptional divisor. At the fixed point corresponding to the vertex $\boxed{k}$ the toric coordinate functions are $u_1=\frac{z_1^{n-k+1}}{z_2^{k-1}}$ and $u_2=\frac{z_2^k}{z_1^{n-k}}$. The monomial $f$ is equal to
$$z_1^{a_1}z_2^{a_2}=u_1^{\tfrac{a_1}nk+\tfrac{a_2}n(n-k)}\,u_2^{\tfrac{a_1}n (k-1)+\tfrac{a_2}n(n-k+1)}\,.$$
Therefore the contribution of the fixed point $\boxed{k}$ to the elliptic class is equal to
$$\delta\left(\tfrac{z_1^{n-k+1}}{z_2^{k-1}},\h^{1-\left(\tfrac{a_1}nk+\tfrac{a_2}n(n-k)\right)}\right)\,
\delta\left(\tfrac{z_2^k}{z_1^{n-k}},\h^{1-\left(\tfrac{a_1}n(k-1)+\tfrac{a_2}n(n-k+1)\right)}\right)\,.$$
Setting $\muu_i=\h^{\tfrac{1-a_i}n}$ we obtain the formula (*).
The second formula describes the orbifold elliptic class by application of \eqref{sympmckay}.
By Theorem \ref{sympmckay} the expressions for the above elliptic classes are equal.

\section{Self-duality of $A_{n-1}$ singularity}
As it has been shown in the subsection \ref{LaGi} mixing the equivariant variables $t_i$ with the variable $\h$ leads to interesting results. The variable $\h$ modified by the parameters depending on the divisor multiplicity sometimes plays a role similar to equivariant variables. In \cite{RW}, \cite{RSVZ} the parameters depending on line bundles are called dynamical parameters and in \cite{AO} -- the K\"ahler variables. It is shown in \cite{RSVZ} that exchanging the equivariant variables with dynamical parameters for elliptic classes of the Schubert varieties in the complete flag variety leads to a mirror self-symmetry. We will show that the $A_n$-singularities are self-symmetric in a similar sense.
\medskip

The formula (**) of Theorem \ref{A1main} clearly does not look like being (anti)symmetric with respect to exchange $h$- and $t$-variables. On the other hand:

\begin{proposition}The expression (*) of Theorem \ref{A1main} is antisymmetric with respect to the change of variables
$$t_1\leftrightarrow \muu_1,\qquad t_2\leftrightarrow \muu_2^{-1}\,.$$
that is
$$A_n(\muu_1,\muu_2^{-1},t_1,t^{-1}_2)=-A_n(t_1,t_2,\muu_1,\muu_2)\,.$$\end{proposition}
\begin{proof}Each summand of (*) after the substitution takes the form
\begin{multline*}\Del\left( {\muu_1^{n-k+1}}{\muu_2^{k-1}},\tfrac{t_1^{k}}{t_2^{n-k}}\right)
\Del\left(\tfrac1{\muu_1^{n-k}\muu_2^{k}},\tfrac{t_1^{k-1}}{t_2^{n-k+1}}\right)=
\\=-\Del\left( {\muu_1^{n-k+1}}{\muu_2^{k-1}},\tfrac{t_1^{k}}{t_2^{n-k}}\right)
\Del\left({\muu_1^{n-k}\muu_2^{k}},\tfrac{t_2^{n-k+1}}{t_1^{k-1}}\right)=
\\=-\Del\left( \tfrac{t_1^{k}}{t_2^{n-k}},{\muu_1^{n-k+1}}{\muu_2^{k-1}}\right)
\Del\left(\tfrac{t_2^{n-k+1}}{t_1^{k-1}},{\muu_1^{n-k}\muu_2^{k}}\right)\,.
\end{multline*}
Setting $k'=n-k+1$ we obtain exactly the summand of the original formula with the minus sign.
\end{proof}

\begin{corollary}We have
\begin{multline*} \sum_{k=0}^{n-1}
\sum_{\ell=0}^{n-1}\big(\tfrac{\muu_2}{\muu_1}\big)^\ell\Del\left(\e\left(\tfrac{k-\ell\tau}n\right)t_1,\muu_1^n\right)
\Del\left(\e\left(\tfrac{\tau\ell-k}n\right)t_2,\muu_2^n\right)
=\\=
 -\sum_{k=0}^{n-1}
\sum_{\ell=0}^{n-1}(t_1t_2)^{-\ell}\Del\left(\e\left(\tfrac{k-\ell\tau}n\right)\muu_1,t_1^n\right)
\Del\left(\e\left(\tfrac{\tau\ell-k}n\right)\muu_2^{-1},t_2^{-n}\right)\,.\end{multline*}
\end{corollary}
This effect does not appear in general. It is due to a special form of the action of the  torus on the resolution of $A_{n-1}$ singularity. For example the symplectic subtorus acts via the same character along the exceptional divisors.

\section{Hirzebruch class --- the limit with $q\to 0$}
The limit $q\to 0$ corresponds to  $\tau\to i\infty$.
We set $y=\e(z)=\h^{-1}$.
Then $$\lim_{\tau\to i\infty}\frac{\theta(\upsilon+\nuu\tau-z)}{\theta(\upsilon+\nuu\tau)}=\begin{cases}y^{-1/2}& \text{ if } -1<\nuu<0\\
y^{-1/2}\,\frac{1-y\e(-\upsilon)}{1-\e(- \upsilon)}& \text{ if }\nuu=0\\
y^{1/2}& \text{ if } 0<\nuu<1\,,\end{cases}$$
or equivalently  $$\lim_{q\to 0}
\frac{\vt(t\,\h\, q^\nuu)}{\vt(t\, q^\nuu)}
=\begin{cases}\h^{1/2}& \text{ if } -1<\nuu<0\\
\h^{1/2}\,\frac{1-\h^{-1}t^{-1}}{1-t^{-1}}& \text{ if }\nuu=0\\
\h^{-1/2}& \text{ if } 0<\nuu<1\,,\end{cases}$$
see \cite[proof of Prop.~3.13]{BoLi} for the proof of the first two limits, and use (\ref{per2}) to deduce the third one.
\medskip

Let  $T_j=t_j^{-1}$ for $j=1,2$.
The equality of unreduced equivariant elliptic classes of $A_{n-1}$  singularity with $D_X=\emptyset$  specializes to the equality of rational functions
{\def\Dellim#1#2{\frac{1-y \tfrac{#2}{#1}}{1- \tfrac{#2}{#1}}}
$$(*)_\infty=y^{-1}\sum_{k=1}^{n}\Dellim{t_1^{n-k+1}}{t_2^{k-1}}
\Dellim{t_2^{k}}{t_1^{n-k}},\\
$$
}
$$(**)_\infty=y^{-1}\frac 1n\sum_{k=0}^{n-1}\frac{1-y\,\e(\tfrac kn) t_1^{-1}}{1-\,\e(\tfrac k n) t_1^{-1}}\cdot\frac{1-y\,\e(-\tfrac k n) t_2^{-1}}{1-\,\e(-\tfrac k n) t_2^{-1}}+(n-1)\,.$$
By elementary transformations the sum $(*)_\infty$ can be written as
$$(***)_\infty=y^{-1}(1 - y) (1 -
    y(t_1t_2)^{-1} ) \frac{1 - (t_1 t_2)^{-n}}{(1 - (t_1 t_2)^{-1}) (1 - t_1^{-n}) (1 - t_2^{-n})} +n \,.$$
Setting $t_1=t_2=T^{-1}$ we obtain a trigonometric identity
$$\frac 1n\sum_{k=0}^{n-1}\frac{1-2\cos(\frac{2k\pi }n) y\, T+y^2T^2}{1-2\cos(\frac{2k\pi }n) T+T^2}=(1-y)\left(1-y
   T^2\right)\frac{1-T^{2 n}}{\left(1-T^2\right)
   \left(1-T^{n}\right)^2}+y\,.$$
In particular for $y=0$
$$\frac 1n\sum_{k=0}^{n-1}\frac{1}{1-2\cos(\frac{2k\pi }n) T+T^2}=\frac{1-T^{2 n}}{\left(1-T^2\right)
   \left(1-T^{n}\right)^2}\,.$$
The equality of expressions $(*)_\infty$, $(**)_\infty$ and $(***)_\infty$ was already noticed in \cite[\S5.5]{DBW}.


\end{document}